\documentclass{amsart}
\usepackage{amsmath}
\usepackage{amssymb}
\usepackage{amsthm}
\usepackage{epsfig}
\usepackage{amsfonts}
\usepackage{amscd}
\usepackage{mathrsfs}
\usepackage{enumerate}
\usepackage{latexsym}

\newtheorem{theorem}{Theorem}[section]

\newtheorem{corollary}[theorem]{Corollary}

\theoremstyle{definition}

\theoremstyle{remark}

\theoremstyle{question}

\newtheorem*{theorem*}{Theorem}

\begin{document}

\title[Embedding normed linear spaces into $C(X)$]{Embedding normed linear spaces into $\mathbf{C(X)}$}

\author[M. Fakhar]{M. Fakhar}
\address{\textbf{M. Fakhar} Department of Mathematics, University of Isfahan, Isfahan 81745--163, Iran, and, School of Mathematics, Institute for Research in Fundamental Sciences (IPM), P.O. Box: 19395--5746, Tehran, Iran.}
\email{fakhar@math.ui.ac.ir}

\author[M.R. Koushesh]{M.R. Koushesh$^*$}
\address{\textbf{M.R. Koushesh} Department of Mathematical Sciences, Isfahan University of Technology, Isfahan 84156--83111, Iran, and, School of Mathematics, Institute for Research in Fundamental Sciences (IPM), P.O. Box: 19395--5746, Tehran, Iran.}
\email{koushesh@cc.iut.ac.ir}

\author[M. Raoofi]{M. Raoofi}
\address{\textbf{M. Raoofi} Department of Mathematical Sciences, Isfahan University of Technology, Isfahan 84156--83111, Iran, and, School of Mathematics, Institute for Research in Fundamental Sciences (IPM), P.O. Box: 19395--5746, Tehran, Iran.}
\email{raoofi@cc.iut.ac.ir}

\thanks{$^*$Corresponding author}




\begin{abstract}
It is well known that every (real or complex) normed linear space $L$ is isometrically embeddable into $C(X)$ for some compact Hausdorff space $X$. Here $X$ is the closed unit ball of $L^*$ (the set of all continuous scalar-valued linear mappings on $L$) endowed with the weak$^*$ topology, which is compact by the Banach--Alaoglu theorem. We prove that the compact Hausdorff space $X$ can indeed be chosen to be the Stone--\v{C}ech compactification of $L^*\setminus\{0\}$, where $L^*\setminus\{0\}$ is endowed with the supremum norm topology.
\end{abstract}

\maketitle

\section{Introduction}

Throughout this note by a \textit{space} we will mean a topological space, unless we explicitly state otherwise. blue The field of scalars (which is fixed throughout discussion) is either the real field $\mathbb{R}$ or the complex field $\mathbb{C}$, and is denoted by $\mathbb{F}$.

For a compact Hausdorff space $X$, we denote by $C(X)$ the set of all continuous scalar-valued mappings on $X$. The set $C(X)$ is a normed linear space when equipped with the supremum norm and pointwise addition and scalar multiplication.

It is known that every (real or complex) normed linear space $L$ can be isometrically embedded into $C(X)$ for some compact Hausdorff space $X$. Here $X$ is the closed unit ball of $L^*$ (the set of all continuous scalar-valued linear mappings on $L$) endowed with the weak$^*$ topology, which is known to be compact by the Banach--Alaoglu theorem. In this note we give a new proof of this well known fact, with $X$ being chosen as the Stone--\v{C}ech compactification of $L^*\setminus\{0\}$, where $L^*\setminus\{0\}$ is endowed with the supremum norm topology. Our proof is rather topological and makes use of some elementary properties of the Stone--\v{C}ech compactification. We conclude with a result which provides an upper bound for the density of $X$ in terms of the density of $L$.

Recall that a \textit{compactification} of a completely regular space $X$ is a compact Hausdorff space which contains $X$ as a dense subspace. The \textit{Stone--\v{C}ech compactification} of a completely regular space $X$, denoted by $\beta X$, is the (unique) compactification of $X$ which is characterized among all compactifications of $X$ by the fact that every continuous bounded mapping $f:X\rightarrow\mathbb{F}$ is extendable to a continuous mapping $F:\beta X\rightarrow\mathbb{F}$. The Stone--\v{C}ech compactification of a completely regular space always exists. For more information on the theory of the Stone--\v{C}ech compactification see \cite{E}, \cite{GJ}, or \cite{PW}.

The Stone--\v{C}ech compactification was introduced independently by M.H. Stone \cite{S} and E. \v{C}ech \cite{C} in 1937, developing an idea of A. Tychonoff \cite{T} (used in the proof of his celebrated result nowadays referred to as the \textit{Tychonoff theorem}). The Banach--Alaoglu theorem was proved by L. Alaoglu \cite{A} in 1940, as a consequence of the Tychonoff theorem; though, a proof of this theorem for separable normed linear spaces had been already published in 1932 by S. Banach \cite{B}. (For an interesting proof of the Banach--Alaoglu theorem assuming the existence of the Stone--\v{C}ech compactification see the recent paper \cite{G} by H. Hosseini Giv.)

\section{The embedding theorem}

Here we prove our embedding theorem. The proof uses only some basic facts from the theory of the Stone--\v{C}ech compactification besides an appeal to the Hahn--Banach theorem.

\begin{theorem}\label{KJJ}
Let $L$ be a normed linear space. Then $L$ can be isometrically embedded into $C(Y)$ for a compact Hausdorff space $Y$, namely, for
\[Y=\beta\big(L^*\setminus\{0\}\big),\]
where $L^*\setminus\{0\}$ is endowed with the supremum norm topology.
\end{theorem}

\begin{proof}
Let $x\in L$. Define
\[\theta_x:L^*\setminus\{0\}\longrightarrow\mathbb{F}\]
by
\[\theta_x(x^*)=\frac{x^*(x)}{\|x^*\|}.\]
It is clear that the mapping $\theta_x$ is continuous, when $L^*\setminus\{0\}$ is endowed with the supremum norm topology. We verify that $\theta_x$ is bounded. For this purpose we indeed show that
\begin{equation}\label{LGD}
\|\theta_x\|=\|x\|.
\end{equation}
Note that
\[\big|\theta_x(x^*)\big|=\frac{|x^*(x)|}{\|x^*\|}\leq\frac{\|x^*\|\|x\|}{\|x^*\|}=\|x\|\]
for any $x^*\in L^*\setminus\{0\}$. Thus $\|\theta_x\|\leq\|x\|$. On the other hand, by the Hahn--Banach theorem, there exists some $z^*\in L^*$ such that $\|z^*\|=1$ and $|z^*(x)|=\|x\|$. Therefore
\[\big|\theta_x(z^*)\big|=\frac{|z^*(x)|}{\|z^*\|}=\|x\|,\]
which implies that $\|x\|\leq\|\theta_x\|$. This shows (\ref{LGD}). Observe that the mapping $\theta_x$, being continuous and bounded, can be extended to the continuous mapping
\[\Theta_x:\beta\big(L^*\setminus\{0\}\big)\longrightarrow\mathbb{F}.\]
Define
\[\Theta:L\longrightarrow C\big(\beta\big(L^*\setminus\{0\}\big)\big)\]
such that
\[x\longmapsto\Theta_x.\]
We show that $\Theta$ embeds $L$ isometrically into $C(\beta(L^*\setminus\{0\}))$, that is, $\Theta$ preserves addition, scalar multiplication and norm. Note that
\begin{equation}\label{JGHD}
\theta_{x+z}=\theta_x+\theta_z
\end{equation}
by definition. Thus
\[\Theta(x+z)=\Theta(x)+\Theta(z),\]
as $\Theta(x+z)$ and $\Theta(x)+\Theta(z)$ are continuous and agree on the dense subspace $L^*\setminus\{0\}$ of $\beta(L^*\setminus\{0\})$ by (\ref{JGHD}); indeed
\[\Theta_{x+z}|_{L^*\setminus\{0\}}=\theta_{x+z}\;\;\;\;\mbox{ and }\;\;\;\;(\Theta_x+\Theta_z)|_{L^*\setminus\{0\}}=\theta_x+\theta_z.\]
Similarly, we can show that
\[\Theta(\alpha x)=\alpha\Theta(x).\]
To conclude the proof we need to show that $\Theta$ preserves norm. It is clear that
\[\|\theta_x\|\leq\|\Theta_x\|,\]
as $\Theta_x$ extends $\theta_x$. Also
\[\|\Theta_x\|\leq\|\theta_x\|,\]
as
\begin{eqnarray*}
|\Theta_x|\big(\beta\big(L^*\setminus\{0\}\big)\big)&=&|\Theta_x|\big(\overline{L^*\setminus\{0\}}\big)\\&\subseteq&\overline{|\Theta_x|\big(L^*\setminus\{0\}\big)}=
\overline{|\theta_x|\big(L^*\setminus\{0\}\big)}\subseteq\big[0,\|\theta_x\|\big].
\end{eqnarray*}
That is
\[\|\Theta_x\|=\|\theta_x\|.\]
This, together with (\ref{LGD}), proves that
\[\big\|\Theta(x)\big\|=\|x\|.\]
\end{proof}

The following corollary should be known. We derive it here, however, as an immediate consequence of the construction given in Theorem \ref{KJJ}.

Recall that the \textit{density} of a space $X$, denoted by $\mathrm{d}(X)$, is the minimum cardinality of a dense subset of $X$; more precisely
\[\mathrm{d}(X)=\min\big\{|D|:D\mbox{ is dense in }X\big\}.\]
In particular, a space $X$ is separable if and only if $\mathrm{d}(X)\leq\aleph_0$. It is clear that the density of a space is always bounded by its cardinality.

\begin{corollary}\label{JHF}
A non-zero normed linear space $L$ can be isometrically embedded into $C(Y)$ for a compact Hausdorff space $Y$ of density at most $2^{\,\mathrm{d}(L)}$.
\end{corollary}

\begin{proof}
Observe that
\[\mathrm{d}\big(\beta\big(L^*\setminus\{0\}\big)\big)\leq\mathrm{d}\big(L^*\setminus\{0\}\big),\]
as any dense subset of $L^*\setminus\{0\}$ is also dense in $\beta(L^*\setminus\{0\})$, since $L^*\setminus\{0\}$ is dense in $\beta(L^*\setminus\{0\})$. Let $D$ be a dense subset of $L$ of minimum cardinality. Note that $D$ is infinite, since $X$ is so, since $X$ is a non-zero linear space. Note that
\[|L^*|=\big|\{x^*|_D:x^*\in L^*\}\big|,\]
as any continuous scalar-valued mapping on $L$ is determined by its value on the dense subset $D$ of $L$. We have
\[\mathrm{d}\big(L^*\setminus\{0\}\big)\leq\big|L^*\setminus\{0\}\big|\leq|L^*|\leq|\mathbb{R}^D|=|\mathbb{R}|^{|D|}=2^{\,|D|}=2^{\,\mathrm{d}(L)}.\]
Theorem \ref{KJJ} now concludes the proof.
\end{proof}

\section*{\bf Acknowledgements}

The first two authors are supported in part by grants from IPM (No. 93550414 and No. 93030418).

The authors would like to thank the referee for reading the manuscript. The authors also thank an editor for comments.

\end{document}